\documentclass[reqno,12pt,centertags]{amsart}
\usepackage{geometry}
\usepackage{amsmath,amsthm,amscd,amssymb,latexsym,upref,stmaryrd}
\usepackage{graphicx}
\usepackage{hyperref}
\usepackage{enumitem}
\usepackage{xcolor}
\newcommand*{\mailto}[1]{\href{mailto:#1}{\nolinkurl{#1}}}
\usepackage[numbers,sort&compress]{natbib}

\def\th#1{\vspace{1mm}\noindent{\bf #1}\quad}

\newcommand{\beq}{\begin{equation}}
	\newcommand{\eeq}{\end{equation}}
\newcommand{\ba}{\begin{align}}
	\newcommand{\ea}{\end{align}}

\renewcommand{\ln}{\text{\rm ln}}

 \allowdisplaybreaks
\numberwithin{equation}{section}

\newtheorem{theorem}{Theorem}[section]
\newtheorem{lemma}[theorem]{Lemma}
\newtheorem{corollary}[theorem]{Corollary}

\theoremstyle{definition}
\newtheorem{definition}[theorem]{Definition}
\newtheorem{proposition}[theorem]{Proposition}

\newtheorem{remark}{Remark}[section]

\begin{document}

	\title[Higher-order inverse problems]
	{Borg-type theorem for a class of higher-order differential operators}
	
\author[A.~W.~Guan]{AI-WEI GUAN}
\address{Department of Mathematics, School of Mathematics and Statistics, Nanjing University of
	Science and Technology, Nanjing, 210094, Jiangsu, People's
	Republic of China}
\email{\mailto{guan.ivy@njust.edu.cn}}

\author[D.~J.~Wu]{DONG-JIE WU}
\address{Department of Mathematics, School of Mathematics and Statistics, Nanjing University of
	Science and Technology, Nanjing, 210094, Jiangsu, People's
	Republic of China}
\email{\mailto{wudongjie@njust.edu.cn}}

\author[C.~F.~Yang*]{CHUAN-FU Yang*}
\address{Department of Mathematics, School of Mathematics and Statistics, Nanjing University of
	Science and Technology, Nanjing, 210094, Jiangsu, People's
	Republic of China}
\email{\mailto{chuanfuyang@njust.edu.cn}}

\author[N.~P.~Bondarenko]{NATALIA P. BONDARENKO}
\address{S.M. Nikolskii Mathematical Institute, RUDN University, 6 Miklukho-Maklaya St, Moscow, 117198, Russian Federation.}
\email{\mailto{bondarenkonp@sgu.ru}}

	\date{\today}
	\footnote{* Corresponding author} 
	
	\begin{abstract}
		{In this paper, we study an inverse spectral operator for the higher-order differential equation $(-1)^my^{(2m)}+ q y = \lambda y$, where $q \in L^2(0,\pi)$.
		We prove that if $\|q\|_2$ is sufficiently small, the two spectra corresponding to the both Dirichlet boundary conditions and to the Dirichlet-Neumann ones uniquely determine the potential $q$. The result extends the Borg theorem from the second order to all even higher orders.}
	\end{abstract}
	
	\maketitle

\section{Introduction}

This paper deals with an inverse spectral problem for the higher-order operators generated by the differential equation
\begin{equation}\label{1}
	(-1)^my^{(2m)}(x)+ q(x) y(x) = \lambda y(x), \quad x \in (0,\pi),\;\;\;\;m\geq 1,
\end{equation}
with the following boundary conditions:
\begin{align}\label{boud1}
	\text{Dirichlet}: \quad & y^{(2i)}(0)=y^{(2i)}(\pi)=0,\;\;\; 0\leq i\leq m-1, \\ \label{boud2}
	\text{Dirichlet-Neumann}: \quad	& y^{(2i)}(0)=y^{(2i+1)}(\pi)=0,\;\;\; 0\leq i\leq m-1,
\end{align}
where $q(\cdot)$ belongs to the class $L^{2}(0,\pi)$ of square-integrable complex-valued functions on $(0,\pi)$, and $\lambda$ is the spectral parameter.

Inverse problems of spectral analysis, which aim to 
recover coefficient(s) of differential operators from their spectral data, have been a cornerstone of mathematical analysis since the pioneering work of Ambarzumian \cite{Am}, followed by those of Borg \cite{Bor}, Levitan \cite{Lev1}, and Marchenko \cite{Mar}. Among these fundamental results, Borg's theorem \cite{Bor} stands out for its elegant solution to the uniqueness question for the inverse problem for second-order Sturm-Liouville operators. It says that, for $m = 1$, the potential $q$ is uniquely determined by the eigenvalues of the two problems \eqref{1},~\eqref{boud1} and \eqref{1},~\eqref{boud2}.
This theorem not only advanced spectral theory but also laid the groundwork for applications ranging from quantum mechanics to geophysical inversion. Despite its significant influence, Borg's theorem has historically been limited to second-order differential operators, raising the question of whether analogous results hold for higher orders. 

The most complete results in the inverse problem
theory were obtained for the second-order operators (see, e.g., the monographs by Levitan \cite{Lev1}, Marchenko \cite{Mar}, Freiling and Yurko \cite{Fre}, Kravchenko \cite{Kra}, and references therein). Spectral theory of higher-order differential operators causes interest of scholars because of applications in various physical contexts, such as modeling flows of viscous liquids in thin membranes and vibrations of elastic beams (see \cite{Ber, Gre, Mik, Pol, Tu}). Higher-order differential operators are inherently more difficult to investigate than second-order
ones, which makes it complicated to generalize Borg's theorem to $m > 1$. However, in recent years mathematicians have witnessed some advancements in this area. The generalization of Borg's uniqueness result \cite{Bor} to the fourth order was considered by Barcilon \cite{Bar1}. He claimed that two coefficients $p$ and $q$ of the equation
$$
	y^{(4)} - (p(x) y')' + q(x) y = \lambda y
$$
can be uniquely determined by three spectra. However, the proof of uniqueness presented in \cite{Bar1} was subsequently found to be incorrect \cite{Bon} by Bondarenko. A valid proof was later provided by Guan et al.\cite{Guan} under the so-called separation condition.
			
Inverse spectral theory for higher-order differential operators of the form
\begin{equation}\label{high}
	\mathcal{L}y = y^{(n)} + \sum_{k=0}^{n-2}p_k(x)y^{(k)}, \quad n \geq 2,
\end{equation}
has been systematically developed by Yurko using the method of spectral mappings~\cite{Yur}. In the case of piecewise analytic coefficients, the inverse problem was further investigated by Yurko via the method of standard models, as detailed in \cite[Section~4]{Yur3}. For \eqref{high} of even order $n = 2m$, Yurko established a significant uniqueness result: given $2m-2$ coefficients, the remaining coefficient can be uniquely reconstructed using two spectra associated with the boundary conditions
\begin{equation*}
	\begin{cases}
		y^{(j)}(0) = 0, & j = 0,1,\ldots,2m-2 \\
		y^{(k)}(\pi) = 0, & k = 0,1.
	\end{cases}
\end{equation*}
Nevertheless, the corresponding inverse problem for operators with $L^2$-coefficients remains unresolved in this framework.
 
Khachatryan \cite{Kh} obtained a remarkable uniqueness result for the reduced operator
\begin{equation*}
	\mathcal{L}_0y = y^{(2m)} + p_0(x)y
\end{equation*}
with vanishing middle coefficients ($p_k \equiv 0$ for $k=1,\ldots,n-2$). It was shown that, when the potential possesses the symmetry: $p_0(x) = p_0(\pi-x)$ and has a sufficiently small $L^2$-norm, the Dirichlet spectrum alone uniquely determines $p_0$ in $L^2(0,\pi)$. Furthermore, Guan et al. \cite{AW} recently demonstrated that if $\|p_0\|_2$ is sufficiently small, then the Dirichlet spectrum and Dirichlet-Neumann spectrum uniquely determine the potential of the fourth-order differential operator in the non-symmetric case. The present paper further extends the results of \cite{Kh} and \cite{AW} by eliminating the symmetry constraint on $p_0$ through an innovative application of two-spectra reconstruction for arbitrary even order. It is worth mentioning that, in general, two spectra are insufficient for recovering higher-order differential operators. Uniqueness and existence for solution of higher-order inverse spectral problems by a larger amount of spectral data were obtained in \cite{Sa, Be,Mal, Yur, Yur1, Yur2, Yur3} and other studies.

Let us proceed to formulating our main result.
Denote by $L_1(q)$ and $L_2(q)$ the operators generated by equation \eqref{1} with the Dirichlet \eqref{boud1} and the Dirichlet-Neumann \eqref{boud2} boundary conditions, respectively.
 Let $\{{\lambda}_n (q)\}_{n \ge 1}$ and $\{{\mu}_n (q)\}_{n\geq1}$ be the spectra of the operators $L_1(q)$ and $L_2(q)$, respectively.
Denote by $B(0,\varepsilon)$ the ball of radius $\varepsilon>0$ about $0$ in $L^2(0,\pi)$.
Our uniqueness theorem is formulated as follows.
\begin{theorem}\label{thm}
	There exists an $\varepsilon>0$ such that, if $q_1$, $q_2\in B(0,\varepsilon)$, $\lambda_n(q_1)=\lambda_n(q_2)$ and $\mu_n(q_1)={\mu}_n(q_2)$ for all $n\geq 1$, then
	$q_1(x)=q_2(x)$ a.e. on $(0,\pi)$ .
\end{theorem}

The proof of Theorem~\ref{thm} is based on the reconstruction of a Riesz basis, following the approach developed in \cite{AW, Bor, SC1}. However, the details involved will be more difficult than those for the second and fourth orders.
We compose a special Riesz basis by using the eigenfunctions of the four operators $L_k(q_i)$, where $i,k = 1, 2$. To achieve this, we study asymptotic properties of the eigenfunctions of the problems $L_k(q)$, $k = 1, 2$, and derive uniform estimates for the remainder terms. It is worth mentioning that the asymptotics of spectral projections, which imply the eigenfunction asymptotics, for the problem $L_1(q)$ were previously obtained by Polyakov\cite{Pol2, Po}. However, in this paper, we adopt new methods by expanding and estimating the Green functions. This approach yields a more accurate estimation of the eigenfunction, which may have applications in various areas of spectral theory.

This paper is organized as follows. In Sect.~\ref{sec:prelim}, the eigenvalues and eigenfunctions of the unperturbed operators are presented, the basic facts regarding Green functions and Riesz bases are given. Sect.~\ref{sec:esth} is focused on the derivation of the asymptotics for eigenfunctions. In Sect.~\ref{sec:proof}, the proof of Theorem~\ref{thm} is provided.

\section{Preliminaries} \label{sec:prelim}

We introduce some basic definitions and notations.

\subsection{Eigenfunctions and eigenvalues of unperturbed operators}

For $q \in L^2(0,\pi)$ and $n \ge 1$, denote by $\phi_n(x; q)$ and $\psi_n(x; q)$ the normalized eigenfunctions corresponding to the eigenvalues $\lambda_n(q)$ and $\mu_n(q)$ of the operators $L_1(q)$ and $L_2(q)$, respectively.

For zero potential, the eigenfunctions and the eigenvalues can be presented in the explicit form:
\begin{equation*}
	{\phi}_{n}(x;0)=\sqrt{\frac{2}{\pi}} \sin(n x), 
\end{equation*}
\begin{equation*} \label{psia}	        
{\psi}_{n}(x;0)=\sqrt{\frac{2}{\pi}} \sin\bigg(\bigg(n-\frac{1}{2}\bigg) x\bigg)
\end{equation*}
and
\begin{equation*}
	\lambda_n(0)=n^{2m},
\end{equation*}
\begin{equation*} \label{muna}	       
	\mu_n(0)=\bigg(n-\frac{1}{2}\bigg)^{2m}.
\end{equation*}

For the eigenvalues of \eqref{1} under the condition \eqref{boud1}, it was established in \cite{Akh} that if $q\in L^2(0,\pi),$ then 
\begin{equation}\label{lambda}
	\lambda_{n}=n^{2m}+\frac{1}{\pi}\int_{0}^{\pi}q(x)dx-\frac{1}{\pi}\int_{0}^{\pi}q(x)\cos 2nxdx+O\bigg(\frac{1}{n^2}\bigg).
\end{equation}

\subsection{Green functions}
Let $G^i(x,y;\lambda),i=1,2$ be the Green functions associated with \eqref{1} under the boundary conditions \eqref{boud1} and \eqref{boud2}, respectively. By definition, $G^i(x,y;\lambda)$ are the integral kernels of the operators $(\lambda-L_i)^{-1}$. If the Dirichlet eigenvalues $\{\lambda_{j}\}_{j\geq 1}$ and the Dirichlet-Neumann eigenvalues $\{\mu_{j}\}_{j\geq 1}$ are simple, using the eigenfunction expansion, we can easily get (see \cite{Nai, Sc}): 
\begin{align}\label{G}
	G^1(x,y;\lambda)=\sum_{j=1}^{\infty}\frac{\phi_{j}(x)\phi_{j}(y)}{\lambda-\lambda_{j}},\;\;\;\;
	G^2(x,y;\lambda)=\sum_{j=1}^{\infty}\frac{\psi_{j}(x)\psi_{j}(y)}{\lambda-\mu_{j}}.
\end{align}

In the unperturbed case, the Green functions become
\begin{equation}\label{K1}
	K^1(x,y;\lambda)=\frac{2}{\pi}\sum_{n=0}^{\infty}\frac{\sin (nx) \sin (ny)}{\lambda-n^{2m}},
\end{equation}
and
\begin{equation*}
K^2(x,y;\lambda)=\frac{2}{\pi}\sum_{n=0}^{\infty}\frac{\sin \big(\big(n-\frac{1}{2}\big)x\big) \sin\big(\big(n-\frac{1}{2}\big)y\big)}{\lambda-\big(n-\frac{1}{2}\big)^{2m}}.
\end{equation*}

Moreover, for $i=1,2$,
\begin{equation}\label{green fun}
	G^i(x,y;\lambda)=\sum_{j=0}^{\infty}(-1)^{j}G_{j}^i(x,y;\lambda).
\end{equation}
where 
\begin{equation}\label{green1}
	\begin{split}
		G_{0}^i(x,y;\lambda)&=K^i(x,y;\lambda),\\
		G_{j}^i(x,y;\lambda)&=\int_{0}^{\pi}K^i(x,\xi;\lambda)q(\xi)G_{j-1}^i(\xi,y;\lambda)d\xi,\;\;j\geq 1.
	\end{split}
\end{equation}
Notice that if $j\geq 1$, $G_{j}^i(x,y;\lambda)$ also can be written as
\begin{equation*}
	\begin{split}
		G_{j}^i(x,y;\lambda)\!\!=\!\!\!\int_{0}^{\pi}\dots\int_{0}^{\pi} \!\!\!K^1(x,\xi_{1};\lambda)q(\xi_{1})K^1(\xi_{1},\xi_{2};\lambda)\dots q(\xi_{j})K^1(\xi_{j},y;\lambda)d\xi_{1}\dots d\xi_{j}.
	\end{split}
\end{equation*}

\subsection{Riesz bases}
\begin{definition}
	A set of vectors $\{e_n\}_{n\geq0}$ is called a Riesz basis of a separable Hilbert space $\mathcal{H}$ if 
	
	\begin{itemize}\label{Rie}
		\item $\{ e_n \}_{n \ge 0}$ is complete in $\mathcal{H}$, i.e., if $(u,e_n)=0$ for all $n$ and some $u \in\mathcal{H} $, then $u=0$;
		\item the vectors $\{ e_n \}_{n \ge 0}$ are independent, i.e., if $\sum_{n=0}^{\infty}
		c_ne_n=0$ for $\{c_n\}\in l^2$, then $c_n=0$ for all $n$;
		\item there exist strictly positive constants $a$ and $A $, such that
		the following inequality holds for every $u \in \mathcal H$: $$a\|u\|^2\leq\sum_{n=0}^{\infty}|(u,e_n)|^2\leq A\|u\|^2.$$
	\end{itemize}
\end{definition}

An important fact is that a small perturbation of a Riesz basis is also a Riesz basis: 

\begin{proposition}\label{le3}\cite{SC1}
	Let $\left\{e_n\right\}_{n \ge 0}$ be a Riesz basis of a separable Hilbert space $\mathcal{H}$. Let $\{ d_n \}_{n \ge 0}$ be a collection of vectors in $\mathcal{H}$ such that 
	\begin{equation}\label{*}
		\sum_{n=0}^{\infty}\left\|d_n-e_n\right\|^2<\frac{a^2}{A} ,
	\end{equation}
	where $a$ and $A$ are defined in definition \ref{Rie}. Then, the vector set $\left\{d_n\right\}_{n \ge 0}$ forms a Riesz basis of $\mathcal{H}$. 	
\end{proposition}

\begin{remark}
	In particular, if  $\{e_n\}_{n\geq0}$ is an orthonormal basis of $\mathcal{H}$ (which is also a Riesz basis of $\mathcal{H}$), then 
	$$
	\|u\|^2=\sum_{n=0}^{\infty}|(u,e_n)|^2. 
	$$
	Consequently, the condition \eqref{*} in Proposition~\ref{le3} becomes 
	$$
	\sum_{n=0}^{\infty}\left\|d_n-e_n\right\|^2<1.
	$$
\end{remark}

More information about Riesz bases can be found in \cite{Christ}.

\section{Estimates for eigenfunctions} \label{sec:esth}
The Green functions are crucial for our asymptotic estimation of eigenfunctions. In this section, we will present the estimates of the Green functions in the unperturbed and perturbed cases, and then apply them to obtain asymptotic formulas for the eigenfunctions.

\subsection{Estimate for free Green function}

In this subsection, we estimate the Green function for the zero potential. Our proof consists of several technical lemmas.

\begin{lemma}\label{4.1}
	For $x>1$, the following inequality holds:
	$$\int_{0}^{x-1}\frac{1}{x^{2m}-t^{2m}}dt<x^{1-2m} \ln x.$$
\end{lemma}
\begin{proof}
	Let $y=\frac{t}{x}$, then we have
\begin{equation}\label{eq1}
	\int_{0}^{x-1}\frac{1}{x^{2m}-t^{2m}}dt=\frac{1}{x^{2m}}\int_{0}^{x-1}\frac{1}{1-(\frac{t}{x})^{2m}}dt=\frac{x}{x^{2m}}\int_{0}^{1-\frac{1}{x}}\frac{1}{1-y^{2m}}dy.
\end{equation}
Since the function $\dfrac{1}{1-y^{2m}}$ equals $\sum_{n=0}^{\infty}y^{2mn}$ and is analytic on $[0, 1-\frac{1}{x}]$, we get
\begin{equation}\label{eq2}
	\int_{0}^{1-\frac{1}{x}}\frac{1}{1-y^{2m}}dy=\sum_{n=0}^{\infty}\frac{y^{2mn+1}}{2mn+1}\bigg|_{0}^{1-\frac{1}{x}}=\sum_{n=0}^{\infty}\frac{(1-\frac{1}{x})^{2mn+1}}{2mn+1}<\sum_{n=1}^{\infty}\frac{(1-\frac{1}{x})^{n}}{n}=\ln x.
\end{equation}
Using \eqref{eq1} and \eqref{eq2}, one can find
$$\int_{0}^{x-1}\frac{1}{x^{2m}-t^{2m}}dt<x^{1-2m}\ln x.$$
\end{proof}

\begin{lemma}\label{ta2}
	For $x>0$, the following relation holds:
$$	\int_{x+1}^{\infty}\frac{1}{t^{2m}-[(x+1)^{2m}-\frac{1}{2}(x+1)^{2m-1}]}dt \leq(x+1)^{1-2m}\ln x.$$
\end{lemma}
\begin{proof}
	Since $0<\frac{(x+1)^{2m}-\frac{1}{2}(x+1)^{2m-1}}{t^{2m}}<1$, we have $$\frac{1}{1-\frac{(x+1)^{2m}-\frac{1}{2}(x+1)^{2m-1}}{t^{2m}}}=\sum_{p=0}^{\infty}\frac{[(x+1)^{2m}-\frac{1}{2}(x+1)^{2m-1}]^p}{t^{2mp}}.$$
	Consequently,
	\begin{align*}
		&\int_{x+1}^{\infty}\frac{1}{t^{2m}-[(x+1)^{2m}-\frac{1}{2}(x+1)^{2m-1}]}dt\nonumber\\&=\int_{x+1}^{\infty}\frac{1}{t^{2m}}\frac{1}{1-\frac{(x+1)^{2m}-\frac{1}{2}(x+1)^{2m-1}}{t^{2m}}}dt\nonumber\\&=\int_{x+1}^{\infty}\frac{1}{t^{2m}}\sum_{p=0}^{\infty}\frac{[(x+1)^{2m}-\frac{1}{2}(x+1)^{2m-1}]^p}{t^{2mp}}dt.
	\end{align*}
	Putting $J=(x+1)^{2m}-\frac{1}{2}(x+1)^{2m-1}$, we get
	\begin{align*}
		\sum_{p=0}^{\infty}\int_{x+1}^{\infty}\frac{J^p}{t^{2m+2mp}}dx&=\sum_{p=0}^{\infty}J^p\frac{t^{1-(2m+2mp)}}{1-(2m+2mp)}\bigg|_{x+1}^{\infty}\nonumber\\&=(x+1)^{1-2m}\sum_{p=0}^{\infty}\frac{J^p (x+1)^{-2mp}}{2m+2mp-1}\nonumber\\&\leq\frac{(x+1)^{1-2m}}{2m-1}+\frac{(x+1)^{1-2m}}{2m}\sum_{p=1}^{\infty}\frac{J^p (x+1)^{-2mp}}{p}\nonumber\\&\leq\frac{(x+1)^{1-2m}}{2m-1}+\frac{(x+1)^{1-2m}}{2m}\sum_{p=1}^{\infty}\frac{(1-2(x+1)^{-1})^p}{p}\nonumber\\&\leq(x+1)^{1-2m}\ln x.
	\end{align*}
\end{proof}

Consider the following family of contours in the complex $\lambda$-plane:
\begin{equation*}
	\Gamma_{j}=\bigg\{\lambda\in \mathbb{C}:|\lambda-j^{2m}|=\frac{1}{2}j^{2m-1}\bigg\},\;\;\;j\in \mathbb{N}.
\end{equation*}
we have the following lemma:

\begin{lemma}\label{free}
	There exists $N\in \mathbb{N}_{+}$, $C>0$, such that if $\lambda\in \Gamma_j$, $j\geq N+1$, then the series in \eqref{K1} converges uniformly on $(x,y,\lambda)\in (0,\pi)\times (0,\pi)\times\Gamma_j$. Moreover, $\big|K^1(x,y;\lambda)\big|\leq C j^{1-2m}\ln j$.
\end{lemma}
\begin{proof}
	For $\lambda \in \Gamma_j$, consider the series
\begin{equation}\label{series}
	\sum_{n=1}^{\infty}\frac{1}{|\lambda-n^{2m}|}.
\end{equation} 
Divide \eqref{series} into three parts as follows:
\begin{align}
&\text{(I)}\;\;\;\; \sum_{n=1}^{j-2}\frac{1}{|\lambda-n^{2m}|},\nonumber\\
&\text{(II)}\;\;\;\; \frac{1}{|\lambda-(j-1)^{2m}|}+\frac{1}{|\lambda-j^{2m}|}+\frac{1}{|\lambda-(j+1)^{2m}|},\nonumber\\
&\text{(III)}\;\; \sum_{n=j+2}^{\infty}\frac{1}{|\lambda-n^{2m}|}.\nonumber
\end{align}

First, consider part (I). For $\lambda \in \Gamma_j$, we have
\begin{equation}\label{4.4}
	\sum_{n=1}^{j-2}\frac{1}{|\lambda-n^{2m}|}\leq\sum_{n=1}^{j-2}\frac{1}{|(j-1)^{2m}-n^{2m}|}\leq\int_{0}^{j-2}\frac{1}{(j-1)^{2m}-t^{2m}}dt.
\end{equation}
Lemma \ref{4.1} yields that
\begin{equation}\label{4.5}
	\int_{0}^{j-2}\frac{1}{(j-1)^{2m}-t^{2m}}dt\leq (j-1)^{1-2m}\ln(j-1).
\end{equation}
From \eqref{4.4} and \eqref{4.5}, we obtain
\begin{equation}\label{4.6}
	\sum_{n=1}^{j-2}\frac{1}{|\lambda-n^{2m}|}\leq (j-1)^{1-2m}ln(j-1)\leq C_1 j^{1-2m}\ln j,
\end{equation}
where $C_1$ is a positive constant.

Next, consider part (II). For $\lambda \in \Gamma_j$ and sufficiently large $j$, we have
\begin{align}\label{4.7}
	&\frac{1}{|\lambda-(j-1)^{2m}|}+\frac{1}{|\lambda-j^{2m}|}+\frac{1}{|\lambda-(j+1)^{2m}|}\nonumber\\&\leq\frac{1}{\frac{1}{2}(j-1)^{2m-1}}+\frac{1}{\frac{1}{2}j^{2m-1}}+\frac{1}{\frac{1}{2}(j+1)^{2m-1}}\leq 6(j-1)^{1-2m}\leq C_2 j^{1-2m},
\end{align} 
where $C_2$ is a positive constant.

Next, consider part (III). For $\lambda \in \Gamma_j$, by Lemma \ref{ta2} we get
\begin{align*}
	\sum_{n=j+2}^{\infty}\frac{1}{|\lambda-n^{2m}|}&\leq\sum_{n=j+2}^{\infty}\frac{1}{n^{2m}-[(j+1)^{2m}-\frac{1}{2}(j+1)^{2m-1}]}\nonumber\\&\leq\int_{j+1}^{\infty}\frac{1}{x^{2m}-[(j+1)^{2m}-\frac{1}{2}(j+1)^{2m-1}]}dx\nonumber\\
&\leq(j+1)^{1-2m}\ln j.
\end{align*}
This yields
\begin{equation}\label{4.11}
	\sum_{n=j+2}^{\infty}\frac{1}{|\lambda-n^{2m}|}\leq(j+1)^{1-2m}lnj\leq C_3j^{1-2m}\ln j,
\end{equation}
where $C_3$ is a positive constant.

Choosing $C=\text{max}\{C_1,C_2,C_3\}$, for $\lambda \in \Gamma_j$ and sufficiently large $j$, by \eqref{4.6}, \eqref{4.7} and \eqref{4.11}, we have
\begin{equation*}
	\sum_{n=1}^{\infty}\frac{1}{|\lambda-n^{2m}|}\leq C j^{1-2m}\ln j.
\end{equation*}
Hence, the series $$\sum_{n=1}^{\infty}\frac{\sin(nx)\sin(ny)}{\lambda-n^{2m}}$$
converges uniformly on $(x,y,\lambda)\in (0,\pi)\times (0,\pi)\times\Gamma_j$. Therefore, $\big|K^1(x,y;\lambda)\big|\leq C j^{1-2m}\ln j$.
\end{proof}

\subsection{Estimates for perturbed Green function and for eigenfunctions}
 	
We begin with an auxiliary lemma.
 
 	\begin{lemma}\label{one}
 		There exists $N\in \mathbb{N}_{+}$, such that there is only one eigenvalue of $L_1(q)$ in  $ \Gamma_n$, $n\geq N+1$.
 	\end{lemma}
 \begin{proof}
 A direct calculation shows that
 	\begin{equation*}
 		(n+1)^{2m}-n^{2m}\!\!=\!\!(n+1)^{2m-1}+(n+1)^{2m-2}n+(n+1)^{2m-3}n^2+\dots+(n+1)n^{2m-2}+n^{2m-1}.
 	\end{equation*}
 Since
 	\begin{equation*}
 		(n+1)^{2m}-n^{2m}>(n+1)^{2m-1},
 	\end{equation*}
 we have
 	\begin{equation*}
 		(n+1)^{2m}-n^{2m}>\frac{1}{2}(n+1)^{2m-1}+\frac{1}{2}n^{2m-1}.
 	\end{equation*}
 	Therefore, the contours $\Gamma_n$ are disjoint. Taking the asymptotic formula \eqref{lambda} into account, we prove the lemma.
 \end{proof}

The main result of this section is the following theorem on the asymptotics of the eigenfunctions for the problems $L_1(q)$ and $L_2(q)$.

\begin{theorem}\label{prop}
Let  $M>0$. Then, there exists a positive integer $P_1$, depending only on  $M$, such that, for any $q  \in B(0, M)$ and $n>P_1$, there hold
\begin{align}\label{a1}
&    \sup _{x \in[0,\pi]}\bigg| \phi_n(x;q)- \sqrt{\frac{2}{\pi}}\sin (n x)\bigg| \leq  C \frac{\ln n}{n^{2m-1}},\\
		\label{a2}
&	\sup _{x \in[0,\pi]}\bigg| \psi_n(x;q)- \sqrt{\frac{2}{\pi}}\sin \bigg(\bigg(n-\frac{1}{2}\bigg) x\bigg)\bigg| \leq  C \frac{\ln n}{n^{2m-1}}, 
\end{align}	
where $C$ is a positive constant depending only on  $M$.
\end{theorem}
\begin{proof}
We only prove \eqref{a1}. The proof of \eqref{a2} is similar and so omitted.

If $n$ is sufficiently large, by \eqref{green fun}, we have
$$
\frac{1}{2 \pi i} \oint_{\Gamma_n} G^1(x, y ; \lambda) d \lambda=\sum_{j=0}^{\infty}(-1)^j \frac{1}{2 \pi i} \oint_{\Gamma_n} G_j^1(x, y ; \lambda) d \lambda.
$$

Using  \eqref{G}, \eqref{green1} and Lemma \ref{one}, we get
$$
\phi_n(x) \phi_n(y)=\frac{2}{\pi} \sin (n x) \sin (n y)-\sum_{j=1}^{\infty}(-1)^j \frac{1}{2 \pi i} \oint_{\Gamma_n} G_j^1(x, y ; \lambda) d \lambda.
$$
Using \eqref{green1} and Lemma \ref{free}, we obtain 
$$
\left|\sum_{j=2}^{\infty}(-1)^j \frac{1}{2 \pi i} \oint_{\Gamma_n} G_j^1(x, y ; \lambda) d \lambda\right| \leq  Cn^{2m-1} \sum_{j=2}^{\infty} \delta_n^j=75 Cn^{2m-1} \frac{\delta_n^2}{1-\delta_n},
$$
where $\delta_n=C n^{1-2m}\|q\|\ln n$. Choosing a sufficiently large $n$, we get $\delta_n\leq\frac{1}{2}$. Thus, there is a constant $M$ (depending on $\|q\|$) such that
$$
\left|\sum_{j=2}^{\infty}(-1)^j \frac{1}{2 \pi i} \oint_{\Gamma_n} G_j^1(x, y ; \lambda) d \lambda\right| \leq \frac{M\ln n}{n^{2m-1}}.
$$
By \eqref{green1}, we have
$$
\frac{1}{2 \pi i} \oint_{\Gamma_n} G_1^1(x, y ;\lambda) d \lambda=\frac{1}{2 \pi i} \oint_{\Gamma_n} \int_0^\pi K^1(x, \xi ; \lambda) q(\xi) K^1(\xi, y ; \lambda) d \xi d \lambda.
$$
The eigenfunction expansion of $K^1(x, y ;\lambda)$ together with the residue theorem give

$$
\frac{1}{2 \pi i}\! \!\oint_{\Gamma_n}\!\!\! \!G_1^1(x, \!y ; \!\lambda) d \lambda\!\!=\!\!-\frac{2}{\pi}\!\! \int_0^\pi \!\!\!\!\sin n \xi\!\left[K^{1n}\!\left(\xi,\! y ; n^{2m}\right) \sin n x\!+\!K^{1n}\!\left(x, \!\xi ; n^{2m}\right)\! \sin n y\right]\! q(\xi) d \xi,
$$
where 
$$
K^{1n}(x, y, \lambda)=\frac{2}{\pi} \sum_{\substack{j=1 \\ j \neq n}}^{\infty} \frac{\sin (j x) \sin (j y)}{\lambda-j^{2m}}.
$$
It follows that

$$
	\left|\frac{1}{2 \pi i} \oint_{\Gamma_n} G_1^1(x, y ; \lambda) d \lambda\right| \leq \frac{8}{\pi}\|q\| \sum_{\substack{j=1 \\ j \neq n}}^{\infty} \frac{1}{\left|j^{2m}-n^{2m}\right|}\leq C n^{1-2m}\ln n,
$$
  where the last inequality uses the same calculation as in Lemma \ref{free}.\\
Therefore, $$
\phi_n(x) \phi_n(y)=\frac{2}{\pi} \sin (n x) \sin (n y)+O(n^{1-2m}\ln n).
$$
Letting  $x=y=x_n=\frac{\pi}{2n}$,
we have 
\begin{equation}\label{41}
	\bigg|\phi_n^2(x_n)-\frac{2}{\pi}\bigg| \leq C n^{1-2m}\ln n.
\end{equation}
Clearly, there is a $P_1>0$ such that
$$
\bigg|\phi_n^2(x_n)-\frac{2}{\pi}\bigg| \leq \frac{1}{\pi},\;\;\; n \geq P_1.
$$
Thus,
$$
\frac{1}{4} \leq |\phi_n^2(x_n)|.
$$
 We now choose $\phi_n(x)$ so that $\phi_n(x_n)>0$. Therefore,
$$
\frac{1}{2} \leq \phi_n(x_n).
$$
Using the inequality 
$$
|a-b| \leq\left|a^2-b^2\right|
$$
for $a, b \geq 1 / 2$ and
taking \eqref{41} into account, we obtain $$\bigg|\phi_n(x_n)-\sqrt{\frac{2}{\pi}}\bigg|\leq Cn^{1-2m}\ln n.$$
Letting $y=\frac{\pi}{2n}$, we get $$\phi_n(x)\phi_n(x_n)=\frac{2}{\pi} \sin (n x) +O(n^{1-2m}\ln n).
$$
Therefore $$\begin{aligned}
\phi_n(x)&=\frac{1}{\phi_n(x_n)}\bigg(\frac{2}{\pi} \sin (n x) +O(n^{1-2m}\ln n)\bigg)\\&=\sqrt{\frac{\pi}{2}}(1+O(n^{1-2m}\ln n))\bigg(\frac{2}{\pi} \sin (n x) +O(n^{1-2m}\ln n)\bigg)\\&=\sqrt{\frac{\pi}{2}}\sin (n x)+O(n^{1-2m}\ln n).
\end{aligned}$$
\end{proof}

\begin{remark}
	Polyakov \cite{Pol2, Po} derived the asymptotic formulas for the spectral projections of the problem $L_1(q)$, which imply the estimate 
	$$
	\sup _{x \in[0,\pi]} \left| \phi_n(x;q_1)- \sqrt{\frac{2}{\pi}}\sin (n x) \right| =O\left(\frac{1}{n^{2m-3/2}}\right). 
	$$
	We have deduced a more precise estimate \eqref{a1} by employing an alternative approach, which can also be applied to other boundary conditions.
\end{remark}

\begin{corollary}\label{4}
	Let  $M>0$. Then, there exists a positive integer $P_1$, depending only on  $M$, such that, for any $q_1,q_2  \in B(0, M)$ and $n>P_1$, there hold
	\begin{align}\label{a11}
		&    \sup _{x \in[0,\pi]}| \phi_n(x;q_1)\phi_n(x;q_2)- {\frac{2}{\pi}}\sin^2 (n x)| \leq  C \frac{\ln n}{n^{2m-1}},\\
		\label{a22}
		&	\sup _{x \in[0,\pi]}\bigg| \psi_n(x;q_1) \psi_n(x;q_2)- {\frac{2}{\pi}}\sin^2 \bigg(\bigg(n-\frac{1}{2}\bigg) x\bigg)\bigg| \leq  C \frac{\ln n}{n^{2m-1}}, 
	\end{align}	
	where $C$ is a positive constant depending only on  $M$.
\end{corollary}
\begin{proof}
Using Theorem~\ref{prop} and direct calculations, we get	
\begin{align*}
	\nonumber\phi_n(x;q_1)\phi_n(x;q_2)&=\bigg(\sqrt{\frac{2}{\pi}}\sin (n x)+O\bigg(\frac{\ln n}{n^{2m-1}}\bigg)\bigg)\bigg(\sqrt{\frac{2}{\pi}}\sin (n x)+O\bigg(\frac{\ln n}{n^{2m-1}}\bigg)\bigg)\\
	&={\frac{2}{\pi}}\sin^2 (n x) +O\bigg(\frac{\ln n}{n^{2m-1}}\bigg).	
\end{align*}	
The proof of \eqref{a22} is similar. 
\end{proof}

\begin{remark}
	In fact, the estimations
	\begin{align*}
		&    \sup _{x \in[0,\pi]} \left| \phi_n(x;q_1)\phi_n(x;q_2)- {\frac{2}{\pi}}\sin^2 (n x) \right| = O(n^{-2}),\\
		&	\sup _{x \in[0,\pi]}\bigg| \psi_n(x;q_1) \psi_n(x;q_2)- {\frac{2}{\pi}}\sin^2 \bigg(\bigg(n-\frac{1}{2}\bigg) x\bigg)\bigg| = O(n^{-2}) 
	\end{align*}	 
    are sufficient to prove Theorem \ref{thm}.
\end{remark}

\section{Proof of Theorem \ref{thm}} \label{sec:proof}

The proof of Theorem \ref{thm} hinges on the following  observations.

\begin{proposition}\label{le2}\cite{Ko}
The functions $\phi_n(x;\cdot)$ and $\psi_n(x;\cdot)$ are continuous with respect to $q$ at zero uniformly by $x \in [0,\pi]$, that is,
	\begin{align*}
	\sup_{x \in[0,\pi]}| \phi_n(x;q)- \phi_n(x;0)| \rightarrow 0, \quad \sup_{x \in[0,\pi]}| \psi_n(x;q)- \psi_n(x;0)| \rightarrow 0
	\end{align*}
as $q$ tends to zero in $L^2(0,\pi)$.
\end{proposition}

\begin{lemma}\label{riesz}
There exists $\varepsilon>0$ such that, if $q_1,q_2  \in B(0, \varepsilon)$ the system
 $$\big\{1\big\}\cup\bigg\{\sqrt{2\pi}\bigg(\frac{1}{\pi}-\phi_n(x;q_1){\phi}_n(x;q_2)\bigg),\sqrt{2\pi}\bigg(\frac{1}{\pi}-\psi_n(x;q_1){\psi}_n(x;q_2)\bigg)\bigg\}_{n\geq 1}$$  is a Riesz bases in $L^2(0,\pi)$.
\end{lemma}
\begin{proof}
 For $n\geq 1$, denote
 \begin{equation*}
 	e_0=1,
 \end{equation*}
 \begin{equation*}\label{6_1}
 	e_{2n}=\sqrt{\frac{2}{\pi}}\cos (2n x),
 \end{equation*}
 \begin{equation*}\label{6_2}
	e_{2n-1}=\sqrt{\frac{2}{\pi}}\cos ((2n-1) x),
\end{equation*}
and
\begin{equation*}
	d_0=1,
\end{equation*}
\begin{equation*}\label{6_3}
	d_{2n}=\sqrt{2\pi}\bigg(\frac{1}{\pi}-\phi_n(x;q_1)\phi_n(x;q_2)\bigg),
\end{equation*}
\begin{equation*}\label{6_4}
	d_{2n-1}=\sqrt{2\pi}\bigg(\frac{1}{\pi}-\psi_n(x;q_1)\psi_n(x;q_2)\bigg).
\end{equation*}

Note that $\{e_{n}\}_{n\geq 0}=\{1\}\cup\bigg\{\sqrt{\frac{2}{\pi}}\cos (n x)\bigg\}_{n\geq 1}$ is an orthonormal basis in $L^2(0,\pi)$. 

Let $M=1$. By Corollary \ref{4}, there exist $C>0$ and $P_1>0$ such that, for $n>P_1$, we have
$$
\| d_n-e_n \|^2\leq C\left(\frac{\ln n}{n^{2m-1}}\right)^2.
$$  
Choose $P>P_1$ such that
$$
2C\sum\limits_{n=P}^{\infty}\left(\frac{\ln n}{n^{2m-1}}\right)^2< \frac{1}{2}. 
$$
Hence, we have
\begin{align*}
	&\sum_{n=2P+1}^{\infty} \| d_n-e_n \|^2\leq 2C\sum\limits_{n=P}^{\infty}\bigg(\frac{\ln n}{n^{2m-1}}\bigg)^2<\frac{1}{2}.
\end{align*}

Put
$$
\delta=\frac{\sqrt{1+\frac{1}{\sqrt{2\pi P}}}-1}{4}.
$$
By Proposition \ref{le2}, $\phi_n(x;q)$ and $\psi_n(x;q)$ are both continuous dependent on $q$ for $n=\overline{1,P}$. Consequently, there exists $\{\varepsilon_{n}\}_{n=\overline{1,2P}}$ such that,
if $q\in B(0,\varepsilon_{2n})$, $n=\overline{1,P}$, then 
\begin{equation}\label{11}
	\bigg\|\phi_n(x;{q})-\sqrt{\frac{2}{\pi}}\sin (n x)\bigg\|\leq \delta,
\end{equation}
so
\begin{equation}\label{12}
	\|\phi_n(x;q)\|<\delta+\bigg\|\sqrt{\frac{2}{\pi}}\sin (n x)\bigg\|<\delta+1;
\end{equation} 
and, if $q\in B(0,\varepsilon_{2n-1})$, $n=\overline{1,P}$, then
\begin{equation}\label{13}
	\bigg\|\psi_n(x;q)-\sqrt{\frac{2}{\pi}}\sin \bigg(\bigg(n-\frac{1}{2}\bigg) x\bigg)\bigg\|\leq \delta,
\end{equation}
so
\begin{equation}\label{14}
	\|\psi_n(x;q)\|<\delta+\bigg\|\sqrt{\frac{2}{\pi}}\sin \bigg(\bigg(n-\frac{1}{2}\bigg) x\bigg)\bigg\|<\delta+1.
\end{equation}

Choose $\varepsilon_0=\min\{\varepsilon_1, \varepsilon_2,\dots, \varepsilon_{2P}\}$. Let $q_1$, $q_2\in B(0, \varepsilon_0)$. From \eqref{11} and \eqref{12}, we have
\begin{align*}
	&\|d_{2n}-e_{2n}\|\\&=\sqrt{2\pi}\|\phi_n(x;q_1)\phi_n(x;q_2)-\phi_n(x;0)\phi_n(x;0)\|\\
	&\leq\sqrt{2\pi}(\|\phi_n(x;q_2)\|\|\phi_n(x;q_1)-\phi_n(x;0)\|\!+\!\|\phi_n(x;0)\|\|\phi_n(x;q_2)-\phi_n(x;0)\|)\\
	&\leq\sqrt{2\pi}(\delta+1)\big(\|\phi_n(x;q_1)-\phi_n(x;0)\|+\|\phi_n(x;q_2)-\phi_n(x;0)\|\big)\\
	&=\frac{1}{2\sqrt{P}}.
\end{align*}
Analogously, from \eqref{13} and \eqref{14}, we obtain
\begin{align*}
	&\quad\|d_{2n-1}-e_{2n-1}\|\leq \frac{1}{2\sqrt{P}}.
\end{align*}
Thus, if $q_1$, $q_2\in B(0, \varepsilon_0)$, we have
\begin{align*}
	\sum_{n=0}^{2P} \| d_n-e_n \|^2<\sum_{n=1}^{2P} \frac{1}{{4P}}=\frac{1}{2}.
\end{align*}
Set $\varepsilon=\min\{\varepsilon_0, 1 \}$. For $q_1$, $q_2\in B(0,\varepsilon)$, the above calculation implies that
\begin{align*}
	\sum_{n=0}^{\infty} \| d_n-e_n \|^2<1.
\end{align*}
By Proposition \ref{le3}, the sequence $\{d_{n}\}_{n\geq 0}$ is a Riesz basis in $L^2(0,\pi)$. 

\end{proof}

\begin{proof}[Proof of the Theorem \ref{thm}]
Substituting eigenfunctions into equation \eqref{1}, we get the relations
\begin{equation}\label{111}
	(-1)^m\phi_{n}^{(2m)}(x,q_1)+ q_
	1(x) \phi_{n}(x,q_1) = \lambda_n(q_1) \phi_{n}(x,q_1),
\end{equation}
\begin{equation}\label{222}
	(-1)^m\phi_{n}^{(2m)}(x,q_2)+ q_
	2(x) \phi_{n}(x,q_2) = \lambda_n(q_2) \phi_{n}(x,q_2).
\end{equation}
Multiplying \eqref{111} by $\phi_{n}(x,q_2)$ and \eqref{222} by $\phi_{n}(x,q_1)$, subtracting the two equations, and taking the relation
\begin{equation}\label{q}
\lambda_n(q_1)={\lambda}_n(q_2) \text { for all } n,
\end{equation}
into account, we obtain
\begin{equation*}
	(\!-1)^m(\phi_{n}^{(2m)}\!(x,q_1)\phi_{n}(x,q_2)- \phi_{n}^{(2m)}(x,q_2)\phi_{n}(x,q_1))\!\!=\!\! (q_2(x)-q_1(x))\phi_{n}(x,q_1)\phi_{n}(x,q_2).
\end{equation*}
Integrating the both sides from $0$ to $\pi$ by parts and using the boundary conditions \eqref{boud1}, we derive
\begin{equation} \label{eqlan}
\int_{0}^{\pi}(q_1(x)-q_2(x))\phi_{n}(x,q_1)\phi_{n}(x,q_2)dx=0.
\end{equation}
Analogously, since
	\begin{equation*} 
	\mu_n(q_1)={\mu}_n(q_2) \text { for all } n,
	\end{equation*}
we have
	\begin{align}\label{44}
	\int_{0}^{\pi}(q_1(x)-q_2(x))\psi_{n}(x,q_1)\psi_{n}(x,q_2)dx=0.
	\end{align}
Using \eqref{lambda} and \eqref{q}, we derive
\begin{align*}
	\frac{1}{\pi}\int_{0}^{\pi}(q_1(x)-q_2(x))dx-\frac{1}{\pi}\int_{0}^{\pi}(q_1(x)-q_2(x))\cos (2n x)dx+O\bigg(\frac{1}{n^2}\bigg)=0.
\end{align*}
Let $n\rightarrow\infty$ in the above relation. Since $q_1(x)$, $q_2(x)\in L^{2}(0,\pi)$, then, according to Riemann-Lebesgue's theorem, we get
\begin{equation}\label{int}
	\int_{0}^{\pi}(q_1(x)-q_2(x))dx=0.
\end{equation}
From $\eqref{eqlan}$ and $\eqref{int}$, we obtain
\begin{equation}\label{9}
		\sqrt{2\pi}\int_0^\pi({q_1(x)-q_2(x)})\bigg(\frac{1}{\pi}-\phi_n(x;q_1){\phi}_n(x;q_2)\bigg)dx=0.
\end{equation}
Similarly, the relations $\eqref{44}$ and $\eqref{int}$ imply
\begin{equation}\label{99}
	\sqrt{2\pi}\int_0^\pi({q_1(x)-q_2(x)})\bigg(\frac{1}{\pi}-\psi_n(x;q_1){\psi}_n(x;q_2)\bigg)dx=0.
\end{equation}
 Thus, by Lemma \ref{riesz}, we have $$\big\{1\big\}\cup\bigg\{\sqrt{2\pi}\bigg(\frac{1}{\pi}-\phi_n(x;q_1){\phi}_n(x;q_2)\bigg),\sqrt{2\pi}\bigg(\frac{1}{\pi}-\psi_n(x;q_1){\psi}_n(x;q_2)\bigg)\bigg\}_{n\geq 1}$$ is a Riesz basis in $L^2(0,\pi)$. Therefore, the relations \eqref{int}, \eqref{9}, and \eqref{99} imply that $q_1(x)=q_2(x)$ a.e. on $(0,\pi)$. 
\end{proof}

\vspace{1mm}
\th{Conflict of Interest} {\rm The authors declare no conflict of interest.}
\par\vspace{1mm}

\th{Acknowledgements} {\rm Ai-wei Guan and Chuan-fu Yang were supported by the National Natural Science Foundation of China (Grant No. 11871031) and by the Natural Science Foundation of the Jiangsu Province of China (Grant No. BK 20201303).} 


\begin{thebibliography}{99}
	\bibitem{Am} Ambarzumian, V. Über eine frage der eigenwerttheorie. Zeitschrift für Physik, 1929, 53(9): 690-695.
	
	\bibitem{Akh} Akhmerova, E. F., Asymptotics of the spectrum of nonsmooth perturbations of differential operators of order 2$m$. Mathematical Notes, 2011, 90: 813-823.
	
	
	\bibitem{Bar1} Barcilon, V., On the uniqueness of inverse eigenvalue problems. Geophysical Journal International, 1974, 38(2): 287-298.
	
	\bibitem{Be} Beals, R., Deift, P., Tomei, C., Direct and Inverse Scattering on the Line. American Mathematical Society, 1988.
	
	\bibitem{Ber} Bernis, F., Peletier, L. A., Two problems from draining flows involving thirdorder ordinary differential equations. SIAM Journal on Mathematical Analysis, 1996, 27(2): 515-527.

	\bibitem{Bor} Borg, G., Eine umkehrung der Sturm-Liouvilleschen eigenwertaufgabe. Acta Mathematica, 1946, 78(1): 1-96.

    \bibitem{Bon} Bondarenko, N. P., Counterexample to Barcilon’s uniqueness theorem for the fourth-order inverse spectral problem. Results in Mathematics, 2024, 79: 183.
    

	
	\bibitem{Christ} Christensen, O., An Introduction to Frames and Riesz Bases.
	Birkh\"auser, Boston, 2003.
    
    \bibitem{Fre} Freiling, G., Yurko, V., Inverse Sturm-Liouville Problems and Their Applications. Nova Science Publishers, Huntington, NY, 2001.
    
    \bibitem{Gre} Gregu$\check{\text{s}}$, M., Third Order Linear Differential Equations. Springer, Dordrecht,
    1987.
    
    


	\bibitem{AW} Guan, A. W., Yang, C. F., Bondarenko, N. P., Borg-type theorem for a class of fourth-order differential operators. arXiv: 2412.07413 [math.SP].
	
	\bibitem{Guan} Guan, A. W., Yang, C. F., Bondarenko, N. P., Solving Barcilon's inverse problems by the method of spectral mappings. Journal of Differential Equations, 2025, 416: 1881-1898. 

	\bibitem{Kh} Khachatryan, I. G., Reconstruction of a differential equation from the spectrum. Functional Analysis and Its Applications, 1976, 10(1): 83-84.
	
	\bibitem{Ko} Kong, Q., Wu, H., Zettl, A., Dependence of eigenvalues on the problem. Mathematische Nachrichten, 1997, 188: 173-201.
	
	\bibitem{Kra} Kravchenko, V. V., Direct and Inverse Sturm-Liouville Problems. Birkh$\ddot{a}$user,
	Cham, 2020.
	
	\bibitem{Lev1} Levitan, B. M., Inverse Sturm-Liouville Problems. VNU Sci. Press, Utrecht,
	1987.
	
	\bibitem{Mal} Malamud, M. M., Similarity of Volterra Operators and Related Questions of the Theory of Differential Equations of Fractional Order. Transactions of the Moscow Mathematical Society, 1994.
	
	\bibitem{Mar} Marchenko, V. A., Sturm-Liouville Operators and Their Applications.
	Birkh\"auser, Basel, 1986.
	
	\bibitem{Mik} Mikhlin, S. G., Chambers, L. I. G., Variational Methods in Mathematical
	Physics. Oxford: Pergamon Press, 1964.
	
	\bibitem{Nai} Naimark, M. A., Linear Differential Operators, 2nd ed., Nauka, Moscow, 1969; English transl. of 1st ed., Part I, Ungar, New York, 1967.

	\bibitem{Pol2} Polyakov, D. M., Spectral properties of an even-order differential operator. Differential Equations, 2016, 52(8): 1098-1103.
	
    \bibitem{Po} Polyakov, D. M., Spectral analysis of an even order differential operator with square integrable potential. Mathematical Methods in the Applied Sciences, 2023, 46(5): 5483-5504.
    
     \bibitem{Pol} Polyakov, D. M., Spectral estimates for the fourth-order operator with matrix coefficients. Computational Mathematics and Mathematical Physics, 2020,
    60: 1163-1184.
    
    \bibitem{Tu} Tuck, E. O., Schwartz, L. W., A numerical and asymptotic study of some third-order ordinary differential equations relevant to draining and coating flows.
    SIAM review, 1990, 32(3): 453–469.
	
	
    \bibitem{Sa} Sakhnovich, L. A., Transformation operator method for higher-order equations, Matematicheskii Sbornik. Novaya Seriya, 1961, 55(97) no. 3, 347--360 [in Russian].
    
    \bibitem{Sc} Schueller, A., Eigenvalue Asymptotics for Self-Adjoint, Fourth-Order, Ordinary Differential Operators. Ph.D. Thesis, Department of Mathematics, University of Kentucky, 1996, http://carrot.whitman.edu/phd thesis.html.
	
	\bibitem{SC1} Schueller, A., Uniqueness for near-constant data in fourth-order inverse eigenvalue problems. Journal of Mathematical Analysis and Applications, 2001, 258(2): 658-670. 
	
	\bibitem{Yur} Yurko, V. A., Method of Spectral Mappings in the Inverse Problem Theory, Inverse and Ill-Posed Problem Series. VSP, Utrecht, Netherlands, 2002.
	
	\bibitem{Yur1} Yurko, V. A., On higher-order differential operators with a singular point. Inverse Problems, 1993, 9(4): 495-502.
	
	\bibitem{Yur2} Yurko, V. A., On higher-order differential operators with a regular singularity. Sbornik Mathematics, 1995, 186(6): 901-928.	
	
	\bibitem{Yur3} Yurko, V. A., Inverse problems of spectral analysis for differential operators and their applications. Journal of Mathematical Sciences, 2000, 98(3): 319-426.	
	\end{thebibliography}
\end{document}